\documentclass[a4 paper, 12pt]{article}
\usepackage{hyperref}
\usepackage{amssymb,amsthm,amsmath,makeidx,verbatim,latexsym,amsfonts,mathrsfs}
\usepackage{hyperref}
\usepackage{graphicx}
\usepackage[none]{hyphenat}
\usepackage{tabularx}
\usepackage{lmodern}
\usepackage{cleveref}
\usepackage{amsmath,amsfonts,amssymb,mathrsfs,amsthm}
\usepackage{multirow}
\usepackage[english]{babel}
\usepackage[all]{xy}
\usepackage{tikz}
\addto{\captionsenglish}{%

}
\usepackage{amssymb}
\usepackage{amscd}
\usepackage{amsmath}
\DeclareMathAlphabet{\mathpzc}{OT1}{pzc}{m}{it}
\newcommand*\rfrac[2]{{}^{#1}\!/_{#2}}
\theoremstyle{plain}
\newtheorem{theorem}{Theorem}[section]
\newtheorem{lemma}[theorem]{Lemma}

\newtheorem{definition}[theorem]{Definition}

\begin{document}
\thispagestyle{empty} \baselineskip 24pt
\begin{center}
A FIXED POINT THEOREM ON FUZZY LOCALLY CONVEX SPACES\\
\ \\
$M.E.\; EGWE^1\; \hbox{and}\;R.A.\; OYEWO^2$\\
Department of Mathematics, University of Ibadan, Ibadan, Nigeria.\\
$^1murphy.egwe@ui.edu.ng\;\;\; ^2raoyewo@gmail.com$\\
\end{center}
\begin{center} \textbf{Abstract}\end{center}	
 Let $X$ be a linear space over a field $\mathbb{K}$ and $(X, \rho, *)$  a fuzzy seminorm space where $(\rho, *)$ a fuzzy seminorm   with $*$ a continuous $t$-norm. We give a fixed point theorem for Fuzzy Locally Convex Space.\\
\textbf{keywords:} \emph{Fixed point, Fuzzy locally convex space, Spherically complete}

\section{Introduction}
The concept of fuzzy vectors, fuzzy topological spaces were introduced and well elucidated by Kastaras in his famous works \cite{Kastaras1},\cite{Kastaras2} and \cite{Kastaras3}. Other invariants of these abound in literature \cite{Bag}. Sadeqi and Solaty Kia \cite{Sadeqi} considered fuzzy seminormed spaces with an example of one, which is fuzzy normable but is not classical normable. More general properties and results on fuzzy seminorms can be seen in \cite{Kastaras2}.

The importance and applications of fixed point theorem cannot be overemphasized. Athaf, \cite{Althaf} established a fixed point theorem on a fuzzy metric spaces while Egwe \cite{Egwe} proved the existence of a fixed point on a nonarchimedean fuzzy normed space. A modern approach to fuzzy analysis is can be seen in \cite{Yeol}.

In this paper, we establish a version of fixed point theorem given by Sehgal in \cite{Sehgal} and in fact prove that there exists a unique fixed point for a spherically complete fuzzy locally convex space.
\section{Main Result}
\begin{definition}\cite{Yager}\cite{Klement}
	\normalfont
	A triangular norm, $t-$norm for short is a mapping  $ *:[0,1]\times[0,1] \longrightarrow[0,1],$ where $*$ is a binary operation
	such that the following axioms are satisfied. $\forall\;u,v,w\in [0,1],$
	\begin{enumerate}
		\item[\textbf{(i)}] $*(u,v) = *(v,u)$
		\item[\textbf{(ii)}] $*(u, *(v,w)) = *(*(u,v),w)$
		\item[\textbf{(iii)}] $*(u,v)\leq *(u,w)~~~~\mbox{  where  } v\leq w.$		
\item[\textbf{(iv)}]
	$	*(u, 1) = u*1 = u $,~~
	$	*(u, 0) = u*0 = 0$
	\end{enumerate}
\end{definition}
 The following  $t$-norms are well-known and frequently used.
\begin{enumerate}
	\item[(1)] $~~~ u*v = \min(u,v)$ (Standard intersection)
	\item[(2)] $~~~ u*v = uv$ (Algebraic product)
	\item[(3)] $~~~ u*v = \max(0,u+v-1)$ (Bounded difference)
\end{enumerate}
In this paper, we shall adopted the first option above.
\begin{definition}\cite{Sadeqi}
		\normalfont
Let $Y$ be a vector space over a field $\mathbb{K}, *$ a continuous $t$-norm. A fuzzy seminorm on $Y$ is a mapping
$p: Y\times\mathbb{R}\longrightarrow [0,1]$
satisfying:
\begin{enumerate}
	\item[(i)]  $p(y,t) = 0$ when $t\leq 0$,	
	\item[(ii)]  $
	p(y,t) = p\left(vy,\frac{t}{|v|}\right) \mbox{  when  }$ $t > 0, v\neq 0$
	\item[(iii)]	$p(y+z, t+s)\geq p(y,t)*p(z,s)$	$t,s\in\mathbb{R},~~ y,z\in Y$
	\item[(iv)] $p(y,\cdot)$ is an increasing function of $\mathbb{R}$ and $\displaystyle\lim_{t\to\infty} p(y,t) = 1.$
\end{enumerate}
Then $(p, *)$ is a fuzzy seminorm. Hence $(Y, p, *)$ is a fuzzy seminorm space.\\
\end{definition}
\begin{definition}
	A family $\mathscr{P}$ of fuzzy seminorms on $Y$ is called separating if to each $y_\circ \neq 0$ there is least one $p\in \mathscr{P}$ and $t \in \mathbb{R}$ such that  $p(y,t)\neq 1 $
\end{definition}
\begin{definition}
	Let $\mathfrak{D}$ be a separated fuzzy locally convex topological vector space, $\mathfrak{A}$ a nonempty subset of $\mathfrak{D}$ and $\mathcal{B}$ be a neighbourhood basis of the origin consisting of absolutely fuzzy convex open subsets of $\mathfrak{D}$. For each $B\in\mathcal{B}$, let $\varphi_B$ be the Minkowski's functional of $B$ and $p$ a fuzzy seminorm on $\mathfrak{A}$. For each $y,z \in \mathfrak{A }$ , $t \in \mathbb{R}$ and $\alpha\in (0,1),$ we have\\
  $\varphi_B(y-z) = inf\{t > 0 :p(y-z)<t\}$.\\
  $\varphi_B(y-z,t)=  sup\{\alpha \in(0,1) :p(y-z)<t \}$.\\
  $	B(0,\alpha,t) = \{y-z: p(y-z,t)>1-\alpha\}$.\\
  $	B(y,\alpha,t) = \{z: p(y-z,t)>1-\alpha\}$.
\end{definition}

\begin{definition}
 A mapping $F:\mathfrak{A}\longrightarrow \mathfrak{D}$ is a fuzzy $B-$ contraction $(B\in\mathcal{B})$ if and only if for each $\varepsilon > 0,~\alpha \in (0,1) $ there is a $\delta = \delta(\varepsilon,B,\alpha) >0 $ and $\beta = \beta(\varepsilon,B,\alpha)\in (0,1)$ such that if  $y,z\in \mathfrak{A}$ and if
\begin{align}\label{*1}
1-\alpha\geq\varphi_B(y-z,\varepsilon + \delta)>1-(\alpha + \beta )~\mbox{  then  }~ \varphi_B(F(y)-F(z),\varepsilon )>1-\alpha
\end{align}
	\end{definition}
If $F: \mathfrak{A}\longrightarrow \mathfrak{D}$ is a fuzzy $B$- Contraction for each $B\in\mathcal{B}$, then $F$ is a fuzzy $\mathcal{B}$- Contraction.\\
Note that if $F$ is a fuzzy $\mathcal{B}$- Contraction, then $F$ is fuzzy continuous.\\
\begin{lemma}
	Let $F:\mathfrak{A}\longrightarrow \mathfrak{D}$ be a fuzzy $\mathcal{B}-$contraction. Then $F$ is fuzzy $\mathcal{B}-$contractive, that is for each $B\in\mathcal{B},\;	\varphi_B(F(y)-F(z),\varepsilon)>\varphi_B(y-z,\varepsilon + \delta)$ if $\varphi_B(y-z,\varepsilon + \delta)\neq 1$ and 1 otherwise.
\end{lemma}
\begin{proof}
		Let $y,z\in \mathfrak{A}$ and suppose $\varphi_B = \varphi$,~$\varphi(y-z,\varepsilon+\delta) = 1-\alpha < 1$ for $\varepsilon > 0 $ and $\alpha\in(0,1) $ . Then $
	\varphi(y-z,\varepsilon + \delta)>1-(\alpha + \beta) ~\mbox{  for each  }~ \delta> 0 $
	and in particular $\varphi(y-z,\varepsilon + \delta_0)>1-(\alpha + \beta_0)$
	where $\delta_0 = \delta(\varepsilon,B,\alpha)$, $\beta_0 = \beta(\varepsilon,B,\alpha)$ .Therefore by \eqref{*1}~~$\varphi(F(y) - F(z),\varepsilon)>1-\alpha$. Since $B$ is open, this implies that $
	\varphi(F(y) - F(z),\varepsilon)>1-\alpha = \varphi(y-z,\varepsilon+\delta)$.
	If $1-\alpha = 1$, then $\varphi(y-z,\varepsilon+\delta)>1-\alpha $~ for each~ $\varepsilon >0$ , $\alpha\in (0,1)$ and hence by \eqref{*1}~~$\varphi(F(y)-F(z),\varepsilon)> 1 - \alpha$ which implies that $\varphi(F(y)-F(z),\varepsilon )=1.$
\end{proof}
\begin{theorem}
	Let $\mathfrak{A}$ be a sequentially complete fuzzy subset of $\mathfrak{D}$, $\mu$ be the membership function on $\mathfrak{A}$ and $F:\mathfrak{A}\longrightarrow \mathfrak{D}$ be a fuzzy $\mathcal{B}$-contraction. If $F$ satisfies the condition:\\
	for each $y\in \mathfrak{A}$ , $\alpha \in (0,1),~ \mu(y)=\alpha$ with $\mu(F(y))> \alpha $,  \mbox{there is a} $\mu_{((y, F(y))\Lambda  \mathfrak{A})}(w) =  \mu _{(y, F(y))}(w) \star \mu_\mathfrak{A} (w)$
	such that  $\mu (F(w))\leq \mu (w)$
	then $F$ has a unique fixed point in $\mathfrak{A}.$
\end{theorem}
\begin{proof}
	Let $y_0\in \mathfrak{A}, t>0$ , $\alpha \in (0,1)$ with $\mu(y_0)= \alpha $ and choose a sequence  $\mu_{y_n}(y_{n_i})\leq \mu_A(y_{n_i})~ \forall~ y_{n_i}\in \mathfrak{D},~ i\in I$ defined (inductively) as follows: for each $n\in I$ (positive integers) If $\mu(F(y_0)\leq \mu(y_0)$ then set $(y_1) = F(y_0).$ Hence $\mu(y_1)\leq \mu(y_0)$ which implies $\varphi (y_1 - y_0,t) \longrightarrow 1$. That is, $y_1- y_0 \longrightarrow 0$
	and if $\mu(F(y_0)) > \mu(y_0)$, let $\mu_{((y_0,F(y_0))\Lambda \mathfrak{A})}(y_1) = \mu_{(y_0,F(y_0))}(y_1)~ \star~ \mu_\mathfrak{A}(y_1)$ such that $\mu(F(y_1))\leq \mu(y_1).$ which implies $\varphi(F(y_1)-y_1,t)\longrightarrow 1$. That is, $F(y_1)-y_1 \longrightarrow 0$ \\
	Since we have chosen the sequence $\{y_n\}$, if $\mu F(y_n)\leq \mu(y_n) $, set $y_{n+1} = F(y_n)$. Hence $\mu(y_{n+1})\leq \mu(y_n)$ implies $\varphi(y_{n+1}-y_n,t)\longrightarrow 1$ and if $\mu (F(y_n)) > \mu(y_n) $, let $\mu_{((y_n,F(y_n))\Lambda \mathfrak{A})}(y_{n+1}) = \mu_{(y_n,F(y_n))}(y_{n+1})~ \star~ \mu_\mathfrak{A}(y_{n+1})$, such that $\mu(F(y_{n+1}))\leq \mu(y_{n+1}).$ which implies $\varphi(F(y_{n+1}) - y_{n+1},t)\longrightarrow 1$ , That is, $F(y_{n+1}) - y_{n+1} \longrightarrow 0$. \\ It then follows that for each $n\in I$, there is a $\lambda_n \in [0,1)$ satisfying
	\begin{equation}\label{*3}
	y_{n+1} = \lambda_n y_n + (1-\lambda_n)F(y_n).
	\end{equation}
	We show that the fuzzy sequence $\{y_n\}$ so constructed satisfies
	\begin{equation}\label{*4}
	(a)~~	y_{n+1}-y_n\longrightarrow 0~~~(b)~~y_n-F(y_n)\longrightarrow 0
	\end{equation}
	
		To establish \eqref{*4},note that by \eqref{*3}
	\begin{align}
	y_{n+1}-y_n& = (1-\lambda_n)(F(y_n)-y_n)\label{*5}\\
	F(y_n)-y_{n+1}& = \lambda_n(F(y_n)-y_n)\label{*6}
	\end{align}
	Therefore, for $B \in\mathcal{B}$ with $\varphi_B = \varphi$, it follows by the above lemma that
	\small
	\begin{align*}
	\varphi(F(y_{n+1})-y_{n+1}, \epsilon)&\geq  \varphi\left(F(y_{n+1})-y_{n+2},\epsilon\right)\star\varphi\left(y_{n+2}-y_{n+1},\epsilon\right)\\
	& \geq \varphi\left(F(y_n)-y_{n+1},\epsilon\right)\star\left(y_{n+1}-y_n,\epsilon\right)\\
		& \geq \varphi\left(\lambda_n(F(y_n)-y_n),\epsilon\right)\star\left((1-\lambda_n)(F(y_n)-y_n,\epsilon\right)\\
	&\geq 1 \star\varphi (F(y_n)-y_n,\epsilon)  \\
	&\geq\varphi(F(y_n)-y_n,\epsilon)
	\end{align*}
	\normalsize
		Thus by \eqref{*5}, $	\varphi(F(y_{n+1})-y_{n+1}, \epsilon)\geq \varphi(F(y_n)-y_n,\epsilon)$
	for each $n\in I$, that is,
	$\{\varphi(F(y_n)-y_n,\epsilon)\}$ is an increasing
	sequence of non negative reals and hence for each $\varphi = \varphi_B, B\in\mathcal{B}$ there is an $ r>0$ and $0<\alpha<1 $ with 
	\begin{equation}\label{*7}
	1-\alpha\geq \varphi(F(y_n)-y_n,r)\longrightarrow 1-\alpha\leq 1  
	\end{equation}
	we claim that $1 - \alpha \equiv 1$. Suppose  $1 - \alpha > 1$.   
	Choose $\delta = \delta(r, B,\alpha)>0$ and $\beta= \beta (r,B,\alpha)\in (0.1)$ satisfying \eqref{*1}. Then by \eqref{*7} there is a $n_0\in I$
	such that $\varphi(F(y_n)-y_n,r+\delta)>1-(\alpha+\beta)$ for all $ n \geq n_0 $.
		  Now choose an $m\in I$, $m\geq n_0 \ni$
	 $ y_{m+1} = F(y_m)$,
	 (let $m = n_0$ if $\mu(F(y_{n_0}))\leq \mu(y_{n_0}), \alpha \in (0,1)~ with ~\mu(y_m)  = \alpha$) otherwise let $m = n_0+1,$ then $\mu(F(y_{n_0+1}))\leq \mu(y_{n_0+1}) $.
	 Thus for this $m$,
	 $$
	 \varphi(y_m-y_{m+1},r+\delta) = \varphi(y_m-F(y_m),r+\delta)>1-(\alpha + \beta)
	 $$
	 and hence by \eqref{*1}
	 $$
	 \varphi(y_{m+1}-F(y_{m+1}),r) = \varphi(F(y_m)-F(y_{m+1}),r)>1-\alpha
	 $$
	 which contradicts \eqref{*7}\\
	  Thus $1 - \alpha = 1$ for each $B\in \mathcal{B}$ and this implies that the sequence $y_n-F(y_n)\rightarrow 0$. This establishes (b) and (a) now  follow by \eqref{*5}
	  	$y_{n+1}-y_n = (1-\lambda_n)(F(y_n)-y_n)$ and since it is a known fact that $F(y_n)$ is shifting towards $y_n$. then as $\lambda_n\longrightarrow 1, y_{n+1}\longrightarrow y_n$ and since $y_n-F(y_n)\longrightarrow 0$ we are sure $\lambda_n$ is moving to 1 hence we can conclude that $y_{n+1}-y_n\longrightarrow 0$.\\
	We assert that $\{y_n\}$ is a Cauchy sequence in $A$. Suppose not.
	Let for each $i\in I,$ $A_i = \{y_n:n\geq i\}$. Then by assumption there is $B\in\mathcal{B}\ni$
	$
	\varphi(y_n-y_m,\varepsilon + \delta)\leq 1-(\alpha + \beta) \mbox{  for any } i\in I.
	$
	Choose an $\varepsilon$ with $0<\varepsilon<1$, $0<\alpha<1$ and a $\delta$ with $0<\delta<\delta(\varepsilon,B, \alpha)$, $0<\beta<\beta(\varepsilon,B,\alpha)<1 $  satisfying $\varepsilon+\delta<1$, $\alpha+\beta<1$. \\
	It follows that $\varphi( y_n-y_m,\varepsilon +\frac{\delta}{2} )\leq 1- (\alpha+\frac{\beta}{2})$ for any $i\in I$. Thus for each $i\in I$, there exist integers $n(i)$ and $m(i)$ with $i\leq n(i)<m(i)$ such that
	\begin{equation}\label{*8}
 \varphi(y_{n(i)}-y_{m(i)},(\epsilon+\frac{\delta}{2})) \leq 1- (\alpha+\frac{\beta}{2}).
	\end{equation}
	
Let $m(i)$ be the least integer exceeding $n(i)$ satisfying \eqref{*8}.
$
$
Then by \eqref{*8}
\begin{equation}\label{*9}
\begin{split}
1 - (\alpha + \beta)\geq \varphi(y_{n_i}-y_{m_i},\varepsilon + \delta) &= \varphi (y_{n(i)}-y_{m(i)-1} + y_{m(i)-1}-y_{m(i)},\varepsilon + \delta )\\ &\geq\varphi (y_{n(i)}-y_{m(i)-1},\varepsilon +\frac{\delta}{2} ) \star \varphi(y_{m(i)-1}-y_{m(i)},\frac{\delta}{2}) \\
&\geq 1-(\alpha+\frac{\beta}{2}) \star \varphi(y_{m(i)-1}-y_{m(i)},\frac{\delta}{2})\\
&\geq  1-(\alpha+ \frac{\beta}{2}) \star 1\\
&\geq 1 - (\alpha+ \frac{\beta}{2})\\
& > 1- (\alpha+ \beta)
\end{split}
\end{equation}
Now by \eqref{*4}, there is a $i_0\in I\ni \varphi (y_i-F(y_i), \rfrac{\delta}{4})> 1 - (\rfrac{\beta}{4})$ and $\varphi (y_{i-1}-y_i, \rfrac{\delta}{4})> 1 - (\rfrac{\beta}{4}) $ whenever $ i \geq i_0$,
and hence by \eqref{*9}
 $\varphi(y_{n(i)}-y_{m(i)},\varepsilon+\delta) > 1-(\alpha+\beta)$.
It follows from \eqref{*1} that for all i $\geq i_0$,
$\varphi(F(y_{n(i)})-F(y_{m(i)}),\varepsilon) > 1-\alpha $
However, for all i $\geq i_0$,

		\small
	\begin{align*}
	1-(\alpha+\frac{\beta}{2})>\varphi(y_{n(i)}-x_{m(i)},\epsilon+\frac{\delta}{2})& \geq\varphi\left(y_{n_(i)}-F(y_{n_{(i)}}) ,\frac{\delta}{4}\right)\\
	& \star \varphi\left( F(y_{n(i)})-F(y_{m(i)}), \epsilon\right)\\
	&\star \varphi\left(F(y_{m(i)})-y_{m(i)}, \frac{\delta}{4}\right)\\
	& \geq 1 *\left(1- \alpha\right)  \star 1\\ 
	& \geq 1-\alpha\\
	&>1-\left(\alpha+\frac{\beta}{2}\right)
	\end{align*}
	\normalsize
	which contradicts (\eqref{*8}).
	Thus $\{y_n\}$ is a Cauchy sequence in $A$  and the sequential completeness implies that there is a $U\in A\ni$
	$$
	\lim_{n\to\infty}\varphi(y_n-U,t) = 1~~\forall~~ t>0
	$$
	it is required to check if the limit is unique.\\
	Suppose there exist $V\in A\ni$
	$$
	\lim_{n\to\infty}\varphi(y_n-V,t) = 1~~\forall~~ t>0
	$$
	such that $V\neq U$ then
	\begin{align*}
	\varphi(U-V,t)&\geq\varphi\left(U-y_n,\rfrac{t}{2}\right)\star\varphi\left(y_n-V, \rfrac{t}{2}\right)\\
	&\geq\left(U-U,\rfrac{t}{2}\right)\star\varphi\left(V-V,\rfrac{t}{2}\right) taking~ limit~ as~ n \to \infty\\
	&\geq 1\star1\\
	&\geq1\\
	&=1
	\end{align*}
	which is indicating that the $U$ is same as $V$ hence our  assumption contradict our result $\therefore U = V$\\
	$\Longrightarrow$ Limit  $U$ is unique.\\
	Next to find out if $F$ has a fixed point since $F$ is fuzzy continuous, consider \\
	$ \varphi(y_{n+1} - y_n ,t) = \varphi(F(y_n) -F(y_{n-1}) ,t) \geq \varphi(y_n-y_{n-1} ,t)$\\
	$ \varphi(y_n- y_{n+1} ,t)\geq \varphi(y_{n-1} - y_n ,t)$\\
	Taking limit as $n\to\infty$ , we get
	\begin{align*}
	\varphi(U - F(U), t) & \geq \varphi(U - U, t)\\
	\varphi(U - F(U), t) & \geq 1 \\
	\varphi(U - F(U), t) & =1 ~ from~ (b) ~in~ \eqref{*4}\\
	\implies U = F(U).
	\end{align*}
	Thus, $U$ is a fixed point in $\mathfrak{A}$.
	Hence, the  existence of fixed point in fuzzy locally convex space $\mathfrak{A}$\\
	Since $U = F(U)$\\
	If $q$ is another fixed point in $A$ then $ q = F(q)$\\
	$\implies \varphi (q-F(q),t)=1,~~t>0$\\
	Such that, $q\neq U$\\
	Hence
	\begin{align*}
	1>\varphi(U-q,t)&\geq\varphi\left(U-F(U),\rfrac{t}{2}\right)\star\varphi\left(F(U)-q,\rfrac{t}{2}\right)\\
	&\geq 1\star\varphi\left(U-F(U),\rfrac{t}{4}\right)\star\varphi\left(F(U)-q,\rfrac{t}{4}\right)\\
	&\geq 1\star1\star\varphi\left(U-F(U),\rfrac{t}{8}\right)\star\varphi\left(F(U)-q,\rfrac{t}{8}\right)\\
	&\geq 1\star1\star1\star\varphi\left(U-F(U),\rfrac{t}{16}\right)\star\varphi\left(F(U)-q,\rfrac{t}{16}\right)\\
	&\vdots\\
	&\geq 1\star1\star1\star1\star\cdots\star\varphi\left(F(U)-q,\rfrac{t}{2^j}\right)\\
	& = 1 \mbox{  as  }  j\longrightarrow\infty
	\end{align*}
	$\Longrightarrow U = q \Longrightarrow U$ is a unique fixed point of the fuzzy locally convex space $\mathfrak{A}$ and this complete the theorem on the existence of a fixed point theorem for fuzzy Locally Convex Space.
		\end{proof}

\end{document}